\documentclass[11pt]{amsart}

\usepackage{latexsym,amsmath,amssymb,amsfonts,MnSymbol,epsfig,verbatim}
\usepackage[initials]{amsrefs}
\usepackage[all]{xy}
\usepackage{tikz}
\usetikzlibrary{decorations.pathmorphing}

\providecommand{\CC}{{\mathbb{C}}}
\providecommand{\RR}{{\mathbb{R}}}
\providecommand{\QQ}{{\mathbb{Q}}}

\providecommand{\ZZ}{{\mathbb{Z}}}

\providecommand{\MM}{{\mathcal M}}

\providecommand{\SO}{{\mathrm{SO}(n)}}

\providecommand{\Spinc}{{\mathrm{Spin}^c(n)}}

\providecommand{\Ind}{{\mathrm{Index}}}

\providecommand{\Null}{{\mathcal{N}}}
\providecommand{\Hardy}{{\Null(\bar{D}_\Omega)}}

\providecommand{\ind}{{\mathrm{Index}\,}}
\providecommand{\ch}{{\mathrm{ch}}}
\providecommand{\Td}{{\mathrm{Td}}}

\newtheorem{theorem}{Theorem}
\newtheorem{lemma}[theorem]{Lemma}
\newtheorem{corollary}[theorem]{Corollary}
\newtheorem{proposition}[theorem]{Proposition}

\theoremstyle{definition}

\theoremstyle{remark} 
\newtheorem{remark}[theorem]{Remark}

\numberwithin{equation}{section}

\begin{document}

\title[Index for Toeplitz Operators as Corollary of Bott Periodicity]{The Index Theorem for Toeplitz Operators as a Corollary of Bott Periodicity}
\author{Paul F.\ Baum}
\address{The Pennsylvania State University, University Park, PA, 16802, USA}
\email{baum@math.psu.edu}
\author{Erik van Erp}
\address{Dartmouth College, 6188, Kemeny Hall, Hanover, New Hampshire, 03755, USA}\email{jhamvanerp@gmail.com}


\dedicatory{
Dedicated to the memory of Sir Michael Atiyah.}

\maketitle

\tableofcontents

\section{Introduction}

This is an expository paper about the index of Toeplitz operators,
and in particular  Boutet de Monvel's theorem \cite{Bo79}.
We prove Boutet de Monvel's theorem as a corollary of Bott periodicity,
and independently of the Atiyah-Singer index theorem.

Let $M$ be an odd dimensional closed Spin$^c$ manifold with Dirac operator $D$ 
acting on sections of the spinor bundle $S$.
If $E$ is a smooth $\CC$ vector bundle on $M$,
$D^E$ denotes $D$ twisted by $E$.
The  closure $\bar{D}^E$ of $D^E$ is an unbounded self-adjoint operator on
the Hilbert space $L^2(M,S\otimes E)$ of $L^2$-sections of $S\otimes E$.
$\bar{D}^E$  has discrete spectrum with finite dimensional eigenspaces.
Denote by $L^2_+(M,S\otimes E)$ the Hilbert space direct sum of the eigenspaces of $\bar{D}_E$ for eigenvalues $\lambda\ge 0$.
$P^E_+$ denotes the orthogonal projection 
\[ P^E_+:L^2(M,S\otimes E)\to L^2_+(M,S\otimes E)\]
Suppose that $\alpha$ is an automorphism of $E$,
and $I_S\otimes \alpha$ the resulting automorphism of $S\otimes E$.
$\mathcal{M}_\alpha$ is the bounded invertible operator on $L^2(S\otimes E)$
obtained from $I_S\otimes \alpha$.
The Toeplitz operator $T_\alpha$ is the composition of $\mathcal{M}_\alpha:L^2_+\to L^2$
with $P^E_+:L^2\to L^2_+$,
\[ T_\alpha =P^E_+\mathcal{M}_\alpha : L^2_+(M,S\otimes E)\to L^2_+(M,S\otimes E)\]
The Toeplitz operator $T_\alpha$ is a Fredholm operator (see section \ref{sec2}).

\begin{theorem}\label{Thm}
Let $M$ be an odd dimensional compact Spin$^c$ manifold without boundary.
If $E$ is a smooth $\CC$ vector bundle on $M$,
and $\alpha$ is an automorphsim of $E$,
then
\[  \ind T_\alpha = (\ch(E,\alpha)\cup \Td(M))[M]\]
Here $\ch(E,\alpha)$ is the Chern character of $(E,\alpha)$, $\Td(M)$ is the Todd class of the Spin$^c$ vector bundle $TM$
and $[M]$ is the fundamental cycle of $M$.
\end{theorem}

Our proof of Theorem \ref{Thm} is based on three points:
\begin{itemize}
\item Bott periodicity.
\item Bordism invariance of the index.
\item Invariance of the index under vector bundle modification.
\end{itemize}
The last two points are analytical, and are proved in this paper.
The key topological feature of our proof is Bott periodicity (in its original form).
Our proof does not use $K$-theory or $K$-homology, or cobordism theory, and is independent of the Atiyah-Singer theorem.

In section \ref{sec:BdM} we show  that Theorem \ref{Thm}  implies Boutet de Monvel's theorem.

Special cases of Theorem \ref{Thm}, when $M=S^1$,  were proven by F.\ Noether \cite{No20}, and Gohberg-Krein \cite{GK60}.
Venugopalkrishna \cite{V72} proved the case of Boutet de Monvel's theorem when  $M=S^{2r+1}$.

\subsection*{Acknowledgement}
We thank Magnus Goffeng for carefully reading a first draft of this paper, and pointing out a mistake in one of the proofs.

\section{Todd class and  Chern character}

In this section we  review the characteristic classes that appear  in Theorem \ref{Thm}.
We assume familiarity with Chern and Pontryagin classes (see \cite{MilSta}).

The Todd class of  a  $\CC$ vector bundle is  the characteristic class
\[ \mathrm{Td} = \prod_j \frac{x_j}{1-e^{-x_j}}\]
where $x_j$ are the Chern roots.
The $\hat{A}$ class of an  $\RR$ vector bundle is the characteristic class 
\[ \hat{A} = \prod_j \frac{x_j/2}{\sinh{(x_j/2)}}\]
where  $x_j$ are  the Pontryagin roots.
Due to the power series identity
\[ e^{x/2}\cdot \frac{x/2}{\sinh{(x/2)}} = \frac{x}{1-e^{-x}}\]
for a $\CC$ vector bundle $F$,
\[ \mathrm{Td}(F) = e^{c_1(F)/2} \; \hat{A}(F)\]
where $\hat{A}(F)$ is the $\hat{A}$ class of the underlying $\RR$ vector bundle of $F$.

A  spin$^c$ vector bundle is an $\RR$ vector bundle with given extra structure.
Thus (as an $\RR$ vector bundle) it has an $\hat{A}$ class.
Associated to each spin$^c$ vector bundle $F$ is a $\CC$ line bundle $L_F$.
If $F$ is a $\CC$ vector bundle, $L_F$  is the determinant bundle.
The first Chern class $c_1(F)$ of a spin$^c$ vector bundle
is, by definition, the  Chern class of $L_F$.
The Todd class of a spin$^c$ vector bundle $F$ is defined by the formula,
 \[ \mathrm{Td}(F) = e^{c_1(F)/2} \; \hat{A}(F)\]
 For a spin vector bundle $F$ the associated line bundle is trivial, and so 
 \[ \mathrm{Td}(F) = \hat{A}(F)\]
The Chern character $\mathrm{ch}(E,\alpha)$ of a smooth $\CC$ vector bundle $E$ with smooth automorphism $\alpha$
is as follows.
First assume that $E$ is the trivial bundle $X\times \CC^r$,
and $\alpha$ is a smooth map $\alpha:X\to GL(r, \CC)$.
Then if the dimension of $X$ is $2m+1$, the Chern character is the cohomology class represented by the differential form 
\[ \mathrm{ch}(\alpha) 
= \sum_{j=0}^m  -\frac{j!}{(2j+1)!(2\pi i)^{j+1}} \mathrm{Tr}((\alpha^{-1}d\alpha)^{2j+1})\]
More generally, if $E$ is not trivial, choose a vector bundle $E'$ such that $E\oplus E'$ is trivialized,
and extend $\alpha$ by adding the identity automorphism on $E'$.
Then proceed as above.

\section{Toeplitz operators on the circle}

The simplest case of Theorem \ref{Thm} is:

\begin{theorem}\label{thm:Noether}
For a continuous function $f:S^1\to \CC\setminus \{0\}$
the Fredholm operator $T_f$ has index
\[ \Ind T_f = -\mathrm{winding\; number}\, f\]
\end{theorem}
\begin{proof}
First consider  $f(z)=z^m$ with $m\ge 0$.
Using the orthonormal basis $e_n(z)=z^n$ (with $n\ge 0$) of $L^2_+(S^1)$,
we have $T_{z^m}e_n=e_{n+m}$. Thus.
\[ \dim\mathrm{Ker}\, T_{z^m} = 0\qquad \dim\mathrm{Coker}\, T_{z^m} =m\]
If  $m<0$ we get  $T_{z^m}e_n=0$ for $n=0,\dots, |m|-1$ and $T_{z^m}e_n=e_{n+m}$ otherwise. Then,
\[ \dim\mathrm{Ker}\, T_{z^m} =|m|= -m\qquad \dim\mathrm{Coker}\, T_{z^m} =0\]
In both cases we find $\Ind T_f = -m$.

Now let $f:S^1\to \CC\setminus \{0\}$ be an arbitrary continuous function with winding number $m$.
Then $f$ is homotopic to $z^m$
by a homotopy  $f_t:S^1\to \CC\setminus \{0\}$, $t\in [0,1]$.
Since the map $f\mapsto T_f$, $C(S^1)\to \mathcal{B}(L^2_+)$ is continous,
it follows that $T_{f_t}$ is a norm continuous path of Fredholm operators.
Therefore
\[\Ind\, T_{f}=\Ind\, T_{z^m}=-m=-\mathrm{winding\;number}\, f\]
\end{proof}

When $f$ is smooth, 
\[ -\mathrm{winding\; number}\, f=-\frac{1}{2\pi i}\int_{S^1}f^{-1}df=\mathrm{ch}(f)[S^1]\]

\section{Toeplitz operators on closed manifolds}\label{sec2}

With  notation as in the introduction, in this section we prove that the Toeplitz operator $T_\alpha$ is a Fredholm operator.
First  assume that $E=M\times \CC$ is a trivial line bundle.
The Dirac operator $D$ of $M$ 
acts on $C^\infty$-sections of the spinor bundle $S$,
\[ D: C^\infty(M,S)\to C^\infty(M,S)\]
As above, denote by $L^2_+(M,S)$ the Hilbert space direct sum of the eigenspaces of $\bar{D}$ for eigenvalues $\lambda\ge 0$.
$P_+$ denotes the orthogonal projection 
\[ P_+:L^2(M,S)\to L^2_+(M,S)\subset L^2(M,S)\]
For a continuous function $f\in C(M)$,  $\mathcal{M}_f$ denotes the multiplication operator on $L^2(S)$,
\[ (\mathcal{M}_fu)(x) = f(x)u(x) \qquad u\in L^2(S)\] 

\begin{lemma}
The commutator $[P_+,\mathcal{M}_f]=P_+\mathcal{M}_f-\mathcal{M}_fP_+$ is compact for every continuous function $f$ on $M$.
\end{lemma}
\begin{proof}
If $f$ is a smooth function, then $\mathcal{M}_f$ is a pseudodifferential operator of order zero.
Therefore, the commutator $[P_+,\mathcal{M}_f]$ is a pseudodifferential operator of order $-1$, and hence compact.
Since  $\|\mathcal{M}_f\|=\|f\|_\infty$ we have $\|[P,\mathcal{M}_f]\|\le 2\|f\|_\infty$, so the map $f\mapsto [P_+,\mathcal{M}_f]$ is continuous.
The lemma follows since $C^\infty(M)$ is uniformly dense in $C(M)$ and the space of compact operators is closed in operator norm. 
\end{proof}

Now let $E$ be any smooth hermitian vector bundle on $M$.
For a continous endomorphism $\alpha$ of $E$ (viewed as a topological vector bundle),
$\mathcal{M}_\alpha$ denotes the bounded operator on $L^2(M,S\otimes E)$
determined by $I_S\otimes \alpha$.
Here $\alpha$ is not necessarily an automorphism of $E$.

\begin{lemma}\label{lem:comm2}
The commutator $[P_+^E,\mathcal{M}_\alpha]$ is compact for every continuous endomorphism $\alpha$ of $E$.
\end{lemma}
\begin{proof}
If $E=M\times \CC^r$ is a trivial bundle, then $\mathcal{M}_\alpha$
is a $r\times r$ matrix of multiplication operators by functions.
Then $[P^E_+,\mathcal{M}_\alpha]$ is a matrix of commutators of $P_+$ with functions, each of which is compact by the previous lemma.
Thus $[P_+^E,\mathcal{M}_\alpha]$  is compact when $E$ is a trivial vector bundle. 

In general, choose $F$ such that $E\oplus F$ is a trivial vector bundle.
Then we can take $D^{E\oplus F}=D^E\oplus D^F$, hence $P^{E\oplus F}_+=P^E_+\oplus P^F_+$.
Also $\mathcal{M}_{\alpha\oplus 0} = \mathcal{M}_\alpha\oplus 0$, 
and so $[P_+^{E\oplus F}, \mathcal{M}_{\alpha\oplus 0}]=[P_+^E, \mathcal{M}_\alpha]\oplus 0$.
Since $[P_+^{E\oplus F}, \mathcal{M}_{\alpha\oplus 0}]$ is compact, so is $[P_+^E, \mathcal{M}_\alpha]$.

\end{proof}


\begin{proposition}
If $\alpha, \beta$ are two endomorphisms of $E$, then $T_\alpha T_\beta-T_{\alpha\beta}$ is a compact operator.
Therefore, if $\alpha$ is an automorphism of $E$ then $T_\alpha$ is a Fredholm operator.
\end{proposition}
\begin{proof}
Due to Lemma \ref{lem:comm2}, modulo compact operators we have
\[T_\alpha T_\beta = P\mathcal{M}_\alpha P \mathcal{M}_\beta \sim PP\mathcal{M}_\alpha \mathcal{M}_\beta = P\mathcal{M}_{\alpha\beta}=T_{\alpha\beta}\]
It follows that if  $\alpha$ is an automorphism with inverse $\beta$, then $T_\alpha T_\beta-I$ and $T_\beta T_\alpha-I$ are compact operators.
Therefore $T_\alpha$ is Fredholm by Atkinson's theorem.
\end{proof}

\section{Toeplitz operators on compact manifolds with boundary}\label{sec3}

In this section $\Omega$ is  a compact even dimensional Spin$^c$ manifold with  boundary $M=\partial \Omega$, and $D_\Omega$ is the Dirac operator of $\Omega$, acting on sections in the graded spinor bundle $S^+\oplus S^-$,
\[ D_\Omega: C_c^\infty(\Omega\setminus M, S^+) \to C_c^\infty(\Omega\setminus M, S^-)\] 
We denote by $\bar{D}_\Omega$ the maximal closed extension  of $D_\Omega$.
$\bar{D}_\Omega$ is a closed unbounded Hilbert space operator with domain
\[ \{u\in L^2(\Omega,S^+)\;\mid \; D_\Omega u \in L^2(\Omega, S^-)\}\]
where $D_\Omega u$ is taken in the distributional sense.
Denote the kernel of $\bar{D}_\Omega$ by
\[\Hardy  = \{u\in L^2(\Omega,S^+)\;\mid \; D_\Omega u = 0 \} \]
and let $Q$ be the Hilbert space projection of  $L^2(\Omega, S^+)$
onto the closed linear subspace $\Hardy$,
\[ Q:L^2(\Omega,S^+)\to \Hardy\]
For a continuous function $f\in C(\Omega)$, let $\mathcal{M}_f$ denote the multiplication operator 
\[ \MM_f : L^2(\Omega, S^+\oplus S^-)\to L^2(\Omega, S^+\oplus S^-)\qquad (\MM_fu)(x) = f(x)u(x) \]
$\MM^+_f$ and $\MM_f^-$ are the restrictions of $\MM_f$ to positive and negative spinors respectively.

\begin{proposition}\label{prop:commutator1}
The commutator $[Q, \mathcal{M}^+_f]$ is a compact operator for all $f\in C(\Omega)$.
Moreover, $\mathcal{M}^+_f Q$ is compact if  $f(x)=0$ for all $x\in \partial \Omega$.
\end{proposition}
\begin{proof}
Let $T$ be the bounded operator 
\[T=\bar{D}_\Omega(1+\bar{D}^*_\Omega\bar{D}_\Omega)^{-1/2}\]
and $V$ the partial isometry (with the same kernel as $\bar{D}_\Omega$) determined by the polar decomposition
\[ \bar{D}_\Omega= V|\bar{D}_\Omega|,\qquad |\bar{D}_\Omega| = (\bar{D}^*_\Omega\bar{D}_\Omega)^{1/2}\]
$T$ has the following standard properties:
\begin{itemize}
\item  $T-V$ is a compact operator.
\item $\mathcal{M}_f^+T-T\mathcal{M}_f^-$ is compact for all $f\in C_0(\Omega\setminus M)$.
\item $\mathcal{M}_f^+(I-T^*T)$ is compact for all $f\in C_0(\Omega\setminus M)$.
\end{itemize}
Proposition 1.1 in \cite{BDT89} implies the much stronger property: 
\begin{itemize}
\item $\mathcal{M}_f^+T-T\mathcal{M}_f^-$ is compact for all $f\in C(\Omega)$.
\end{itemize}
Since also $T^*\mathcal{M}_f^+-\mathcal{M}_f^-T^*$ is compact for all $f\in C(\Omega)$,
we obtain that the commutators $[\mathcal{M}_f^+, T^*T]$ are compact.
Note that
\[ \mathrm{ker}\; \bar{D}_\Omega = \mathrm{ker}\; T = \mathrm{ker}\; V \]
Therefore $Q=I-V^*V$ differs from $I-T^*T$ by a compact operator.
Hence $[Q,\mathcal{M}_f^+]$ is compact.
Finally, the third property above implies that $\mathcal{M}^+_fQ$ is compact if $f\in C_0(\Omega\setminus M)$.
(For full details, see \cite{BDT89}.)

\end{proof}

An endomorphism $\theta$ of the trivial vector bundle $E=\Omega\times \CC^r$
is naturally identified with a continuous matrix-valued  function
\[ \theta:\Omega\to M_r(\CC)\] 
The multiplication operator $\MM_\theta$ is the bounded operator on the Hilbert space
\[ L^2(\Omega, S^+\otimes E) = L^2(\Omega, S^+)\otimes \CC^r\]
obtained from $I_{S^+}\otimes \theta$.
Denote $Q_r=Q\otimes I_r$, 
\[ Q_r: L^2(\Omega, S^+)\otimes \CC^r \to \Hardy \otimes \CC^r\]
where $I_r$ is the $r\times r$ identity matrix.

\begin{corollary}\label{comcomp}
For every  continuous map $\theta:\Omega\to M_r(\CC)$,
the commutator $[Q_r, \MM_\theta]$ is a compact operator on $L^2(\Omega, S^+)\otimes \CC^r$.
Moreover, $\MM_\theta Q_r$ is compact for every $\theta$ such that $\theta(x)=0$ 
for all $x\in M$.
Here $Q_r$ is viewed as an operator from $L^2$ to $L^2$. 
\end{corollary}
\begin{proof}
The proof is the same as the proof of Lemma \ref{lem:comm2}.
\end{proof}

The Toeplitz operator $T_\theta$  is the composition of $\MM_\theta$ with $Q_r$,
\[ T_\theta =Q_r\mathcal{M}_\theta : \Hardy\otimes \CC^r \to \Hardy\otimes \CC^r\]

\begin{proposition}
If $\theta, \eta$ are two continuous maps $\Omega\to M_r(\CC)$, then $T_\theta T_\eta-T_{\theta\eta}$ is compact.
If $\theta(x)=0$ for all $x\in M=\partial\Omega$ then $T_\theta$ is compact.
Therefore $T_\theta$ is a Fredholm operator if $\theta(x)$ is an invertible matrix for every $x\in M$.
\end{proposition}
\begin{proof}
The first two statements follow from Proposition \ref{prop:commutator1}.
Now let $\theta:\Omega\to M_r(\CC)$ be a continuous function,
and $\theta(x)$ invertible for every $x\in M$.
Let $\eta(x) = \theta(x)^{-1}$, and extend $\eta$ to a continuous function $\eta:\Omega\to M_r(\CC)$.
Then modulo compact operators $T_{\theta\eta}\sim I$ and so $T_{\theta}T_\eta\sim T_\eta T_{\theta}\sim I$.
\end{proof} 

\begin{remark}
Results  closely related to the content of this section were proved by Venugopalkrishna \cite{V72} in the special case where $\Omega$ is a strongly pseudoconvex domain in $\CC^n$.
For the general case of compact Spin$^c$ manifolds with boundary see \cite{BDT89}, and also \cite{BD91}.
\end{remark}

\section{The trace map and the Calderon projection}\label{sec4}

In this section we show how the index of Toeplitz operators on $\Omega$ is related to the index of Toeplitz operators on $M$.
As in the previous section, $\Omega$ is  a compact even dimensional Spin$^c$ manifold
and $M$ is the boundary of $\Omega$.
$D_\Omega$ is the Dirac operator of $\Omega$, acting on sections of the spinor bundles $S^+$ and $S^-$.
$D_M$ is the Dirac operator of $M$, acting on sections of the spinor bundle $S$.
We identify $S^+|M = S$.

We denote the $L^2$ Sobolev space of degree $s\in \RR$ (on $\Omega$ or on $M$)  by $W_s$.
Note that $W_0=L^2$.
$C^\infty(\Omega, S^+)$ is the space of sections of $S^+$ that are smooth on $\Omega$, 
which means, in particular, that all derivatives up to all orders extend continuously to  the boundary $M$.
For a smooth section $s\in C^\infty(\Omega, S^+)$,
let $\gamma(s)\in C^\infty(M,S)$ denote the restriction of $s$ to $M$.
The restriction map $\gamma$ extends to a bounded linear  map on Sobolev spaces, called the trace map,
\[ \gamma_s:W_{s+\frac{1}{2}}(\Omega, S^+)\to W_{s}(M,S)\qquad s\in \RR\]
Let $W_s^\natural(M,S)$ be the subspace of $W_s(M,S)$ consisting of restrictions to the boundary (via the trace map) of  distributional solutions of $D_\Omega u=0$,
\[ W_s^\natural(M,S) = \{ \gamma_s(u) \mid u \in W_{s+\frac{1}{2}}(\Omega,S^+), \; D_\Omega u=0\}\]
The Calderon projection $P_\natural$ is the orthogonal projection of $L^2$ onto $W^\natural_0$,
\[ P_\natural:L^2(M,S)\to W_0^\natural(M,S)\subset L^2(M,S)\]  
$P_\natural$ is a pseudodifferential operator of order zero.
For all $s\in \RR$,  the range of the idempotent $P_\natural:W_{s}\to W_{s}$ is $W_s^\natural(M,S)$.

\begin{proposition}\label{propa}
The bounded operator $F:=P_\natural(1+D_M^2)^{-1/4}\circ \gamma_{-\frac{1}{2}}$,
\[ F:\Hardy \to W_0^\natural(M,S)\]
is Fredholm.
\end{proposition}
\begin{proof}
The  space of $L^2$-solutions  $u\in L^2(\Omega, S^+)$ of the Dirichlet problem
\[ D_\Omega u=0, \qquad \gamma_{-\frac{1}{2}}(u)=0\]
is finite dimensional.
Therefore the surjective map
\[ \gamma_{-\frac{1}{2}}: \Hardy \to W_{-\frac{1}{2}}^\natural(M,S)\]
is a Fredholm operator.

Denote $A=(1+D_M^2)^{-1/4}$.
$A$ is an elliptic pseudodifferential operator of order $-1/2$.
$A$ has scalar symbol, and so the commutator $[A, P_\natural]$ is of order $-3/2$.
Let $B$ be the pseudodifferential operator of order $-1/2$,
\[ B = P_\natural AP_\natural +(I-P_\natural )A(I-P_\natural )\]
$A-B$ is of order $-3/2$,
\begin{align*}
  A-B&= P_\natural A(I-P_\natural )  +(I-P_\natural )AP_\natural \\
  &= [P_\natural ,A](I-P_\natural ) +(I-P_\natural )[A,P_\natural ]
  \end{align*}
  and so $A$ and $B$ have the same principal symbol.
Thus $B$ is elliptic, and the bounded operator
\[B:W_{-\frac{1}{2}}(M,S)\to L^2(M,S)\]
is Fredholm. Because $B$ commutes with $P_\natural$, $B$ restricts to a bounded operator 
\[ B_\natural: W_{-\frac{1}{2}}^\natural (M,S)\to W_0^\natural(M,S)\]
which is also Fredholm.
Finally,  $F=P_\natural A\circ \gamma_{-\frac{1}{2}}=B_\natural\circ \gamma_{-\frac{1}{2}}$ is the composition of two Fredholm operators.

\end{proof}

For a smooth funtion $\tilde{f}\in C^\infty(\Omega)$,
let $T_{\tilde{f}}=Q\mathcal{M}_{\tilde{f}}$ be the operator
\[ T_{\tilde{f}} : \Hardy \to \Hardy\]
with $Q$ as in section \ref{sec3}.
If $f$ is the restriction of $\tilde{f}$ to $M$, let $T^\natural_f=P_\natural\mathcal{M}_f$ be the operator
\[ T_{f}^\natural : W^0_\natural(M,S)\to W^0_\natural(M,S)\]

\begin{proposition}\label{propb} 
With $F=P_\natural(1+D_M^2)^{-1/4}\circ \gamma_{-\frac{1}{2}}$ as above,
the diagram
\[ \xymatrix{   \Hardy \ar[r]^{T_{\tilde{f}}}\ar[d]_F &  \Hardy \ar[d]^{F}& \\
 W^0_\natural(M,S) \ar[r]^{T_f^\natural} & W^0_\natural(M,S)}
\]
commutes modulo compact operators, i.e. $FT_{\tilde{f}}-T^\natural_f F$ is a compact operator for any $\tilde{f}\in C^\infty(\Omega)$, and $f=\tilde{f}|M$.
\end{proposition}
\begin{proof}
Let $\sim$ denote equality modulo compact operators. By Proposition \ref{prop:commutator1},
\[ FT_{\tilde{f}} = FQ\mathcal{M}_{\tilde{f}} \sim F\mathcal{M}_{\tilde{f}}  Q = F \mathcal{M}_{\tilde{f}}= P_\natural A\gamma_{-\frac{1}{2}}  \mathcal{M}_{\tilde{f}}  \]
 $P_\natural$ and $A$ are pseudodifferential operators, and the principal symbols of $P_\natural$ and $A$ commute with the symbol of $\mathcal{M}_f$.
Therefore  $P_\natural\mathcal{M}_f \sim \mathcal{M}_f P_\natural$ and $A\mathcal{M}_f\sim \mathcal{M}_f A$, and
\[ T^\natural_f F = P_\natural\mathcal{M}_f P_\natural A \gamma_{-\frac{1}{2}} 
\sim  P_\natural A \mathcal{M}_f \gamma_{-\frac{1}{2}}\]
The proposition now follows from the equality  
\[\gamma_{-\frac{1}{2}}  \mathcal{M}_{\tilde{f}} = M_f \gamma_{-\frac{1}{2}}\]
\end{proof}

\begin{corollary} \label{corc}
With $\tilde{f}$, $f$, as above, if $f(x)\ne 0$ for all $x\in M$, then 
$T_{\tilde{f}}$ and $T_f^\natural$ are Fredholm operators, and 
\[ \mathrm{Index}\, T_{\tilde{f}} = \mathrm{Index}\, T_f^\natural\]
\end{corollary}
\begin{proof}
As in section \ref{sec2}, since the projection $P_\natural$ is a pseudodifferential operator of order zero,
$T^\natural_f$ is Fredholm if $f(x)\ne 0$ for all $x\in M$.
Combining Propositions \ref{propa} and \ref{propb} we see that $FT_{\tilde{f}}$ and $T_f^\natural F$ are Fredholm operators with the same index, and therefore
\[ \mathrm{Index}\, F + \mathrm{Index}\, T_{\tilde{f}} = \mathrm{Index}\, T_f^\natural + \mathrm{Index}\, F\]

\end{proof}

The Calderon operator $P_\natural$ and the projection $P_+$ (defined in section \ref{sec2}) are pseudodifferential operators of order zero with the same principal symbol.
This fact, combined with Corollary \ref{corc}, gives the main result of this section.

\begin{proposition}\label{propd}
Let $\theta:M\to M_r(\CC)$ be a continuous matrix-valued function on $M$, such that $\theta(x)$ is invertible for all $x\in M$,
and $\tilde{\theta}:\Omega\to M_r(\CC)$  any continuous extension of $\theta$ to $\Omega$.
Then the Fredholm operators
\[ T_\theta. :L^2_+(M,S)\otimes \CC^r\to L^2_+(M,S)\otimes \CC^r\]
(as in section \ref{sec2}) and
\[ T_{\tilde{\theta}} : \Hardy\otimes \CC^r\to \Hardy\otimes \CC^r\]
(as in section \ref{sec3}) have the same index,
\[ \mathrm{Index}\, T_\theta = \mathrm{Index}\, T_{\tilde{\theta}}\]
\end{proposition}
\begin{proof}
For simplicity, assume first that $r=1$.
The projections $P_\natural$ and $P_+$ differ by a compact operator, because they have the same principal symbol.
This implies that
\[ \mathrm{Index}\, T_\theta = \mathrm{Index}\, T_\theta^\natural\]
The proposition now follows from Corollary \ref{corc}.

For $r>1$, the evident generalizations of Proposition \ref{propa}, Proposition \ref{propb}, and Corollary \ref{corc}  are valid, and Proposition \ref{propd} follows.

\end{proof}

\section{Bordism invariance of the index}

In this section we prove bordism invariance of the index of Toeplitz operators,
based on the results of sections \ref{sec3} and \ref{sec4}.

\begin{proposition}
Let $\Omega$ be a compact even dimensional Spin$^c$ manifold with  boundary $M$.
Let $\tilde{E}$ be a smooth $\CC$ vector bundle on $\Omega$, and $\tilde{\alpha}$ an automorphism of $\tilde{E}$.
Denote by  $(E,\alpha)$ the restriction of $(\tilde{E}, \tilde{\alpha})$  to $M$, 
and by  $T_\alpha=P_+^E\MM_\alpha$ the Toeplitz operator determined by $\alpha$,
\[ T_\alpha : L^2_+(M,S\otimes E)\to L^2_+(M,S\otimes E)\]
Then
\[ \mathrm{Index}\, T_\alpha = 0\]
\end{proposition} 
\begin{proof}
Choose $\tilde{F}$ such that $\tilde{E}\oplus \tilde{F}\cong \Omega\times \CC^r$ is trivial.
Then the Toeplitz operator $T_{\alpha \oplus I_F}$ determined by the automorphism $\alpha\oplus I_F$ 
of $E\oplus F$ has the same index as $T_\alpha$.
Thus, it suffices to prove the proposition for trivial vector bundles $E$. 
Hence we assume that $\tilde{\alpha}$ and $\alpha$ are matrix valued functions,
\[ \tilde{\alpha}:\Omega\to GL(r,\CC) \qquad \alpha: M\to GL(r,\CC)\]
Note that $\tilde{\alpha}(x)$ is invertible for all $x\in \Omega$.
By Proposition \ref{propd}, it suffices to show that 
\[\mathrm{Index}\, T_{\tilde{\alpha}}=0\]
where $T_{\tilde{\alpha}}$ 
is the Toeplitz operator determined by $\tilde{\alpha}$,
\[ T_{\tilde{\alpha}} : \Hardy\otimes \CC^r\to \Hardy\otimes \CC^r\]
There is the direct sum decomposition
\[   L^2(\Omega, S^+) \otimes \CC^r= \mathcal{N}(\bar{D}_\Omega)\otimes \CC^r \oplus \mathcal{N}(\bar{D}_\Omega)^\perp\otimes \CC^r\]
With respect to this decomposition,
the multiplication operator $\mathcal{M}_{\tilde{\alpha}}=\mathcal{M}^+_{\tilde{\alpha}}$ is a $2\times 2$ matrix
\[ \mathcal{M}^+_{\tilde{\alpha}}
=\left(\begin{array}{cc}T_{\tilde{\alpha}}&A\\B&S_{\tilde{\alpha}}\end{array}\right)
\]
where $A$ and $B$ are compact operators by Proposition \ref{comcomp}.
Since $\mathcal{M}_{\tilde{\alpha}}$ is invertible, $T_{\tilde{\alpha}}$ and $S_{\tilde{\alpha}}$ are Fredholm operators, and
\[ \mathrm{Index}\, T_{\tilde{\alpha}} + \mathrm{Index}\, S_{\tilde{\alpha}} = \mathrm{Index}\, \mathcal{M}^+_{\tilde{\alpha}} = 0\]
Thus it will suffice to show that 
\[ \mathrm{Index}\, S_{\tilde{\alpha}} = 0\]
As in section \ref{sec3} above, let $V$ be the partial isometry determined by $\bar{D}_\Omega = V |\bar{D}_\Omega|$,
where $V$ has the same kernel as $\bar{D}_\Omega$.
Denote $V_r=V\otimes I_r$.
Consider the diagram
\[ \xymatrix{   \mathcal{N}(\bar{D}_\Omega)^\perp \otimes \CC^r \ar[r]^{S_{\tilde{\alpha}}}\ar[d]_{V_r} &   \mathcal{N}(\bar{D}_\Omega)^\perp\otimes \CC^r \ar[d]^{V_r}& \\
 L^2(\Omega, S^-)\otimes \CC^r\ar[r]^{\mathcal{M}^-_{\tilde{\alpha}}}& L^2(\Omega, S^-)\otimes \CC^r}
\]
In this diagram $V_r$ is a Fredholm operator.
The diagram commutes modulo compact operators, because
\[ V_rS_{\tilde{\alpha}} = V_r(I-Q_r)\mathcal{M}^+_{\tilde{\alpha}} \sim V_r\mathcal{M}^+_{\tilde{\alpha}}(I-Q_r)=V_r\mathcal{M}^+_{\tilde{\alpha}}\]
where $\sim$ denotes equality modulo compact operators.
Finally
\[ V_r\mathcal{M}^+_{\tilde{\alpha}} \sim \mathcal{M}^-_{\tilde{\alpha}} V_r\] 
follows from 
\[ T\mathcal{M}^+_{\tilde{\alpha}} \sim \mathcal{M}^-_{\tilde{\alpha}} T\]
where $T=\bar{D}_\Omega(I+\bar{D}_\Omega^*\bar{D}_\Omega)^{-1/2}$
(as in the proof of Proposition \ref{prop:commutator1}),
and the fact that $T-V$ is compact.

Thus, the Fredholm operators $V_rS_{\tilde{\alpha}}$ and $\mathcal{M}^-_{\tilde{\alpha}}V_r$ have the same index, and additivity of the index implies
\[ \mathrm{Index}\, S_{\tilde{\alpha}} = \mathrm{Index}\,\mathcal{M}^-_{\tilde{\alpha}}\]
Since  $\mathcal{M}^-_{\tilde{\alpha}}$ is an invertible operator, it has index zero.

\end{proof}



\section{The product lemma}

Let  $M_1$ and $M_2$ be two closed even dimensional spin$^c$ manifolds with Dirac operators $D_1, D_2$.
The Dirac operator of $M_1\times M_2$
is the sharp product,
\[ D_{M_1\times M_2} = D_1 \# D_2  =\left(\begin{matrix}D_1\otimes I & -I\otimes D_2^*\\ I\otimes D_2& D_1^*\otimes I\end{matrix}\right)\]
where the $2\times 2$ matrix is an operator from the positive spinors 
\[ S^+_{M_1\times M_2} = (S_1^+\boxtimes S_2^+)\oplus (S_1^-\boxtimes S_2^-)\]
to the negative spinors
\[ S^-_{M_1\times M_2} = (S_1^+\boxtimes S_2^-)\oplus (S_1^+\boxtimes S_2^-)\]
It is straightforward to derive from this formula that
\[ \mathrm{Index}\, D =  \mathrm{Index}\, D_1\,\cdot \, \mathrm{Index}\, D_2\]
In this section we prove the analogue of this product formula 
in the case when $M_1$ is odd dimensional and $M_2$ is even dimensional.

\begin{proposition}\label{Prod}
Let $M$ be a closed odd dimensional spin$^c$ manifold,
and let $E$, $\alpha$, $\mathcal{M}_\alpha$, $T_\alpha$ be as above.
Let $W$ be a closed even dimensional spin$^c$ manifold, 
and $F$ a smooth $\CC$ vector bundle on $W$.
On the  odd dimensional spin$^c$ manifold $M\times W$,
let $T_{\alpha\otimes I_F}$ be the Toeplitz operator associated to the automorphism $\alpha\otimes I_F$ of the vector bundle $E\boxtimes F$.
Then
\[  \mathrm{Index}\, (T_{\alpha\otimes I_F}) = \mathrm{Index}\, T_\alpha  \,\cdot \, \mathrm{Index}\, (D_{W}\otimes F)\]
where $D_W\otimes F$ is the Dirac operator of $W$ twisted by $F$.
\end{proposition}
\begin{proof}
Let $A=D_M\otimes E$ denote the Dirac operator of $M$ twisted by $E$,
 and $B=D_W\otimes F$  the Dirac operator of $W$ twisted by $F$.
 $B=B^+\oplus B^-$ is graded.
The spinor bundle of $M\times W$ is
 \[ S_{M\times W} = S_M\boxtimes S_W = (S_M\boxtimes S_W^+)\oplus (S_M\boxtimes S_W^-)\]
where $S_M, S_W=S_W^+\oplus S_W^-$ are the spinor bundles of $M, W$.
 
The Dirac operator of $M\times W$ twisted by $E\otimes F$ is
 \[ D = \left(\begin{matrix}A\otimes I & I\otimes B^-\\ I\otimes B^+& -A\otimes I\end{matrix}\right)\]
acting on the direct sum
\[ [(S_M\otimes E) \boxtimes (S_W^+\otimes F)]\oplus [(S_M\otimes E) \boxtimes (S_W^-\otimes F)]\]
Consider the polar decomposition $B=V|B|$, where $V$ is a partial isometry with the same kernel as $B$.
Let $U$ be the operator
 \[ U = \left(\begin{matrix}0& I\otimes -V^-\\ I\otimes V^+& 0\end{matrix}\right)\]
$U$ anticommutes with $D$ and commutes with $\mathcal{M}_{\alpha\otimes I_F}$,
\[UD=-DU\qquad U\mathcal{M}_{\alpha\otimes I_F}=\mathcal{M}_{\alpha\otimes I_F}U\]
 $V$ restricts to a unitary operator on   $(\mathrm{Ker}\, B)^\perp$,
 and thus $U$ restricts to a unitary operator on $L^2(S_M^E) \otimes (\mathrm{Ker}\, B)^\perp$.

View the Hilbert space $L^2((S_M\otimes E)\boxtimes (S_W\otimes F))$ as the direct sum
\[ [L^2(S_M^E) \otimes \mathrm{Ker}\, B^+]\oplus  [L^2(S_M^E) \otimes \mathrm{Ker}\, B^-]\oplus   [L^2(S_M^E) \otimes (\mathrm{Ker}\, B)^\perp]\]
where $S_M^E=S_M\otimes E$.
The multiplication operator $\mathcal{M}_{\alpha\otimes I_F}$ is 
\[ \mathcal{M}_{\alpha\otimes I_F}  =
\left(\begin{matrix}
\mathcal{M}_\alpha \otimes I_{{\mathrm Ker}\, B^+} & 0&0\\ 
0 & \mathcal{M}_\alpha \otimes I_{{\mathrm Ker}\, B^-} & 0\\ 
0& 0& \mathcal{M}_\alpha \otimes I_{({\mathrm Ker}\, B)^\perp}
\end{matrix}\right)\]
The three summands are invariant spaces for $D$, and the positive space of $D$ is 
\[ [L^2_+(S_M^E) \otimes  \mathrm{Ker}\, B^+]\oplus   [L^2_-(S_M^E) \otimes  \mathrm{Ker}\, B^-]\oplus  H\]
where $H$ is a closed linear subspace of $L^2(S_M^E) \otimes (\mathrm{Ker}\, B)^\perp$.
Thus the Toeplitz operator $T_{\alpha\otimes I_F}$ is of the form
\[ T_{\alpha\otimes I_F}  =
\left(\begin{matrix}
T_\alpha \otimes I_{{\mathrm Ker}\, B^+} & 0 & 0\\ 
0 & T^-_\alpha \otimes I_{{\mathrm Ker}\, B^-} & 0\\ 
0 & 0& Q
\end{matrix}\right)\]
where $T^-_\alpha$ is the compression of $\mathcal{M}_\alpha$ to  $L^2_-(S^E_M)$,
and $Q$ is the restriction of $T_{\alpha\otimes I_F}$ to $H$.
Since
\[ \mathrm{Index}\, T^-_\alpha = -  \mathrm{Index}\, T_\alpha\]
it follows that
\begin{align*}
 \mathrm{Index}\, T_{\alpha\otimes I_F} &=  \mathrm{Index}\, T_\alpha \,\cdot\,\mathrm{dim\, Ker}\,B^+ + \mathrm{Index}\, T^-_\alpha\,\cdot\, \mathrm{dim\, Ker}\,B^-+\mathrm{Index}\,Q\\
 & =  \mathrm{Index}\, T_\alpha\,\cdot\, \mathrm{Index}\,B+\mathrm{Index}\,Q
 \end{align*}
 On the third summand $L^2(S_M^E) \otimes (\mathrm{Ker}\, B)^\perp$ we have
 \[ UDU^*=-D\qquad U\mathcal{M}_{\alpha\otimes I_F}U^*=\mathcal{M}_{\alpha\otimes I_F}\]
 Therefore the compression of $\mathcal{M}_{\alpha\otimes I_F}$
 to the positive space (i.e. the operator $Q$) has the same index as the compression to the negative space.
 Hence both are zero, i.e. $\mathrm{Index}\,Q=0$



\end{proof}

\section{Vector bundle modification}

With $M, E, \alpha$ as above, let $F\to M$ be a spin$^c$ vector bundle on $M$ with even fiber dimension $n=2r$.
Part of the Spin$^c$ datum of $F$ is a principal $\Spinc$ bundle $P$ on $M$,
\[F=P\times_\Spinc \RR^n\]
where $\Spinc$ acts on $\RR^n$ via the map $\Spinc\to \SO$.
The Bott generator vector bundle $\beta$ on $S^n\subset \RR^n\times \RR$ is $\Spinc$ equivariant,
where $\Spinc$ acts on the first factor $\RR^n$.
Therefore, associated to $P$ we have a fiber bundle $\pi\,\colon \Sigma F \to M$ whose fibers are oriented spheres of dimension $n$,
\[ \Sigma F = P\times_\Spinc S^n\]
and  a vector bundle on $\Sigma F$
\[ \beta_F = P\times_\Spinc \beta\]
Since $M$ is a an odd dimensional Spin$^c$ manifold,  the total space of $F$ is an odd dimensional Spin$^c$ manifold, because $TF = \pi^*F\oplus \pi^*TM$, 
and the direct sum of two Spin$^c$ vector bundles is Spin$^c$.
Every trivial bundle is Spin$^c$, so the total space of $F\oplus \underline{\RR}$ is Spin$^c$.
$\Sigma F$ is a Spin$^c$ manifold as the boundary of the unit ball bundle of $F\oplus \underline{\RR}$.

On the odd dimensional Spin$^c$ manifold $\Sigma F$ we have a Toeplitz operator $T_{\tilde \alpha}$, 
where $\tilde{\alpha}$ is the automorphism $\pi^*\alpha\otimes I$ of the vector bundle $\pi^*E\otimes \beta_F$.
Here $\pi:\Sigma F\to M$ is the projection.

\begin{proposition}\label{VM}
\[  \mathrm{Index}\, T_\alpha = \mathrm{Index}\,T_{\tilde\alpha}\] 
\end{proposition}

The proof of Proposition \ref{VM} is a straightforward generalization of the proof of Proposition \ref{Prod},
and uses the following basic fact about the Dirac operator of an even dimensional sphere.

\begin{proposition}\label{index1}
If $D$ is the Dirac operator of the even dimensional sphere $S^n$
with the Spin$^c$ structure it receives as the boundary of the unit ball in $\RR^{n+1}$,
then $D_\beta$ has one dimensional kernel and zero cokernel, and so
\[ \mathrm{Index} \;D_\beta = 1\]
Moreover, $D_\beta$ is equivariant for the group Spin$^c(n+1)$ which acts by orientation preserving isometries  on $S^n$, and this group acts trivially on the kernel of $D_\beta$.
\end{proposition}
\begin{proof}
Recall that if $V$ is a finite dimensional vector space, then there is a canonical nonzero element in $V\otimes V^*$,
which maps to the identity map under the isomorphism $V\otimes V^*\cong \mathrm{Hom}(V,V)$. 

On any even dimensional Spin$^c$ manifold $M$, there is a canonical isomorphism of vector bundles
\[ S\otimes S^* \cong \Lambda_\CC TM\]
where $S$ is the spinor bundle of $M$.
This is implied by the fact that, as representations of $\Spinc$,
\[ \CC^{2^r}\otimes (\CC^{2^r})^*\cong \Lambda_\CC \RR^n\]
Via this isomorphism,  $D_{S^*}$ identifies (up to lower order terms) with $d+d^*$,
where $d$ is the de Rham operator, and $d^*$ its formal adjoint.
Note that the kernel of $D_{S^*}$ is the same as the kernel of $D_{S^*}^2=(d+d^*)^2$,
i.e. it consists of harmonic forms.
On $S^n$ the only harmonic forms are the constant functions, and scalar multiples of the standard volume form.

Because the Bott generator vector bundle $\beta$ is dual to $S^+$,
$S^+\otimes \beta\cong \mathrm{Hom}(S^+,S^+)$ contains a trivial line bundle,
which identifies with the line bundle in $\Lambda^0\oplus \Lambda^n$
spanned by $(1,\omega)$, where $\omega$ is the standard volume form.
This follows from  representation theory. 
Thus, the kernel of $D_\beta$ is the one dimensional vector space spanned by the harmonic forms $c+c\omega$, $c\in \CC$.
The intersection of $\Lambda^0\oplus \Lambda^n$ with $S^-\otimes \beta$ is zero,
and so the cokernel of $D_\beta$ is zero.

\end{proof}

\begin{remark} $D_\beta$ is the (positive) half-signature operator of $S^n$ (with $n$ even).
\end{remark}

\begin{proof}[Proof of Proposition \ref{VM}]
If $F$ is a trivial bundle, then  $\Sigma F=M\times  S^n$
and Proposition \ref{VM} is a special case of   Proposition \ref{Prod}.
In general, $\Sigma F$ is a sphere bundle over $M$.
The proof is essentially the same as the proof of Proposition \ref{Prod}, with the following modifications.

$D_\beta$ is equivariant for the action of the structure group of the sphere bundle $\Sigma F$.
Therefore, there is a well-defined ``vertical'' operator $B$ on $\Sigma F$ which in each local trivialization of the sphere bundle $\Sigma F$ is $I\otimes D_\beta$.

Next, we construct a first order differential operator $A$ on $\Sigma F$
which is a ``lift'' of $D_M\otimes E$ from $M$ to $\Sigma F$.
Choose a finite open cover $\{U_j\}$ of $M$,
such that the fiber bundle $\Sigma F$ restricted to each $U_j$ has been trivialized,
\[ \Sigma F|U_j \cong U_j\times S^n\]
Let $A_j$ be the evident lift of $D_M\otimes E$ from $U_j$ to the product $U_j\times S^n$,
\[A_j:C^\infty(\pi^*((S^+\otimes E)|U_j))\to C^\infty(\pi^*((S^-\otimes E)|U_j)) \]
$A_j$ differentiates in the $U_j$ direction.
Let $\{\varphi_j\}$ be a smooth partition of unity on $M$ subordinate to the cover $\{U_j\}$.
Using this partition of unity and the local trivializations $\Sigma F|U_j \cong U_j\times S^n$,
we construct the lift $A$ as
\[ A := \sum A_j(\varphi_j\circ \pi) :C^\infty(\pi^*(S_1^+\otimes E))\to C^\infty(\pi^*(S_1^-\otimes E)) \]
Note that $A$ and $B$ commute.

Now, the sharp product formula for the Dirac operator of $\Sigma F$ twisted by $\pi^*E\otimes \beta_F$ is
 \[ D = \left(\begin{matrix}A & B^-\\ B^+& -A\end{matrix}\right)\]
We can  construct the partial isometry $U$ as in the proof of Proposition \ref{Prod},
with  the properties
\[UD=-DU\qquad U\mathcal{M}_{\alpha\otimes I}=\mathcal{M}_{\alpha\otimes I}U\]
$D$ and $\mathcal{M}_{\alpha\otimes I}$ act on the  Hilbert space $\mathcal{H}=L^2(S_{\Sigma F}\otimes \pi^*E\otimes \beta_F)$.
The Hilbert space $L^2(S_{S^n}\otimes \beta)$ on which $D_\beta$ acts is a direct sum $\mathrm{Ker}D_\beta\oplus(\mathrm{Ker}D_\beta)^\perp$.
The summands are invariant  for the structure group $SO(n+1)$ of the sphere bundle $\Sigma F$.
If we view $\mathcal{H}$ as the Hilbert space of $L^2$-sections in a field of Hilbert spaces over $M$,
the orthogonal decomposition of its fibers gives rise to an orthogonal decomposition $\mathcal{H}=\mathcal{H}_1\oplus \mathcal{H}_2$.
Due to the commuting of $A$ and $B$,
the subspaces $\mathcal{H}_1$ and $\mathcal{H}_2$ are invariant for $D$ as well as for $\mathcal{M}_{\alpha\otimes I}$,
and $U$ restricts to a unitary on $\mathcal{H}_2$.

The Toeplitz operator $T_{\tilde{\alpha}}=T_1\oplus T_2$ has a corresponding direct sum decomposition,
where $T_j$ acts on a subspace of $\mathcal{H}_j$.
As in the proof of Proposition \ref{Prod}, $\mathrm{Index}\,T_2=0$
because  $\mathrm{Index}\,UT_2U^* = -\mathrm{index} \,T_2$.
Due to the fact that  the structure group  acts trivially on $\mathrm{Ker}D_\beta$,
the first summand $\mathcal{H}_1$  identifies canonically with $L^2(S_M\otimes E)$,
and $T_1$ identifies with $T_\alpha$. Thus,
\[ \mathrm{Index} \,T_{\tilde{\alpha}}= \mathrm{Index} \,T_1+ \mathrm{Index} \,T_2= \mathrm{Index} \,T_\alpha\]

\end{proof}

\begin{remark}
The canonical identifications $\mathcal{H}_1=L^2(S_M\otimes E)$ and $T_1=T_\alpha$ in the above proof  are due to the fact that the  kernels of the family of Dirac operators of the fibers of $\Sigma F$, twisted by $\beta_F$, form a trivial line bundle over $M$.

\end{remark}

\section{Bordism and vector bundle modification for the topological index}

Bordism invariance of the topological index follows immediately from Stokes' Theorem.

\begin{proposition}
Let $\Omega$ be a compact even dimensional Spin$^c$ manifold with  boundary $M$.
Let $\tilde{E}$ be a smooth $\CC$ vector bundle on $\Omega$, and $\tilde{\alpha}$ an automorphism of $\tilde{E}$.
Denote by  $(E,\alpha)$ the restriction of $(\tilde{E}, \tilde{\alpha})$  to $M$, 
and by  $T_\alpha=P_+^E\MM_\alpha$ the Toeplitz operator determined by $\alpha$,
\[ T_\alpha : L^2_+(M,S\otimes E)\to L^2_+(M,S\otimes E)\]
Then
\[ (\mathrm{ch}(E, \alpha)\cup \mathrm{Td}(M))[M] = 0\]
\end{proposition} 
\begin{proof}
The restriction  of the Spin$^c$ vector bundle $T\Omega$ to $M$
 is the direct sum of the Spin$^c$ vector bundle $TM$ with a trivial line bundle (i.e. the normal bundle).
The Todd class is a stable characteristic class.
Therefore the cohomology class $\mathrm{Td}(M)$ is the restriction of $\mathrm{Td}(\Omega)$ to $M$.
Likewise, by naturality, $\mathrm{ch}(E, \alpha)$ is the restriction to $M$ of
$\mathrm{ch}(\tilde{E}, \tilde{\alpha})$.
The proposition now follows from Stokes' Theorem.

\end{proof}

Invariance of the topological index under vector bundle modification is
\[   (\mathrm{ch}(\pi^*E\otimes \beta_F, \pi^*\alpha\otimes I)\cup \mathrm{Td}(\Sigma F))[\Sigma F]  =(\mathrm{ch}(E,\alpha)\cup \mathrm{Td}(M))[M]\]
Note that
\[\mathrm{ch}(\pi^*E\otimes \beta_F, \pi^*\alpha\otimes I) = \pi^*\mathrm{ch}(E,\alpha)\cup \mathrm{ch}(\beta_F)\]
As spin$^c$ vector bundles on $\Sigma F=S(F\oplus\underline{\RR})$,
\[ \underline{\RR} \oplus T(\Sigma F) \cong \pi^*F\oplus \underline{\RR} \oplus \pi^*(TM)\]
Multiplicativity of the Todd class implies
\[ \mathrm{Td}(\Sigma F) = \pi^*\mathrm{Td}(M)\cup \pi^*\mathrm{Td}(F)\]
Therefore invariance of the topological index under vector bundle modification follows from
\begin{proposition}\label{IF}
\[\pi_!\,\mathrm{ch}(\beta_F) = \frac{1}{\mathrm{Td}(F)}\]
where $\pi_!$ is integration along the fiber of $\pi\,\colon \Sigma F\to M$.
\end{proposition}
For a proof of Proposition \ref{IF}, see section 6.2 in \cite{BvE1}.

\section{Sphere lemma}

\begin{lemma}\label{sphere}
Let $M, E, \alpha$ be as above.
If $n=\mathrm{dim}\,M$, then for every odd integer $m\ge 2n+1$
there exists a continuous map $\tilde{\alpha}:S^{m}\to \mathrm{GL}(r,\CC)$
such that  
\[ \mathrm{Index}\,T_\alpha = \mathrm{Index}\,T_{\tilde{\alpha}}\]
and 
\[(\ch(E,\alpha)\cup \Td(M))[M]=\ch(\tilde{\alpha})[S^m]\]
\end{lemma}
\begin{remark}
Note that $\mathrm{Td}(S^m)=1$ if $m$ is odd.
\end{remark}
\begin{proof}
We shall prove the Lemma by moving, in a finite number of steps, from $(M,E,\alpha)$ to $(S^m,\underline{\CC}^r,\tilde{\alpha})$.
Each step preserves the analytic index as well as the topological index.
We denote   equality of both the analytic and topological index as an equivalence $\sim$ of triples.

By the Whitney Embedding Theorem, $M$ embeds in $\RR^m$.
The 2-out-of-3 principle implies that the normal bundle $\nu$
\[ 0\to TM\to M\times \RR^m\to \nu\to 0\]
is spin$^c$ oriented.
Vector bundle modification by $\nu$ gives
\[ (M,E,\alpha)\sim (\Sigma\nu, \pi^*E\otimes \beta_\nu, \pi^*\alpha\otimes I)\]
By adding on a vector bundle with its identity automorphism, if necessary,
\[ (\Sigma\nu, \pi^*E\otimes \beta_\nu, \pi^*\alpha\otimes I) \sim (\Sigma\nu,\underline{\CC}^r,\alpha_1)\]
Since $M$ is compact we may assume that 
$M$ is embedded in the interior of the unit ball of $\RR^{m}$. 
Using the inclusion $\RR^{m}\to \RR^{m+1}$, 
a compact tubular neighborhood of $M$ in $\RR^{m+1}$ 
identifies with the ball bundle $B(\nu\oplus \underline{\RR})$, whose boundary is $\Sigma \nu$.
Let $\Omega$ be the unit ball of $\RR^{m+1}$ with the interior of $B(\nu\oplus \underline{\RR})$ removed.
Then $\Omega$ is a bordism of spin$^c$ manifolds from $\Sigma \nu$ to $S^{m}$.

By Lemma \ref{L} below, there exist continuous maps 
$\alpha_2:\Omega\to \mathrm{GL}(r,\CC), \alpha_3:B(\nu\oplus \underline{\RR})\to \mathrm{GL}(r,\CC)$, such that, when restricted to $\Sigma\nu$, $\alpha_3\alpha_1=\alpha_2$.
Then 
\[ \mathrm{Index}\,T^{\Sigma\nu}_{\alpha_2}=\mathrm{Index}\,T^{\Sigma\nu}_{\alpha_3\alpha_1}=\mathrm{Index}\,T^{\Sigma\nu}_{\alpha_3}+\mathrm{Index}\,T^{\Sigma\nu}_{\alpha_1}\]
Since $\Sigma \nu$ is the boundary of $B(\nu\oplus \underline{\RR})$ and $\alpha_3$ is invertible on $B(\nu\oplus \underline{\RR})$,
\[ \mathrm{Index}\,T^{\Sigma\nu}_{\alpha_3}=0\]
Thus
\[ \mathrm{Index}\,T^{\Sigma\nu}_{\alpha_2}=\mathrm{Index}\,T^{\Sigma\nu}_{\alpha_1}\]
The same argument applies to the topological index, and we obtain
\[ (\Sigma\nu,\underline{\CC}^r,\alpha_1)\sim(\Sigma\nu,\underline{\CC}^r,\alpha_2)\]
Finally, the bordism $(\Omega,\alpha_2)$ gives
\[(\Sigma\nu,\underline{\CC}^r,\alpha_2)\sim(S^m,\underline{\CC}^r,\alpha_2)\]

\end{proof}

\begin{lemma}\label{L}
Let the unit ball in $\RR^n$ be the union of two compact subsets $B, \Omega$ with  $\Sigma=B\cap \Omega$.
Given a continuous map $\alpha:\Sigma\to \mathrm{GL}(r,\CC)$, there exist continuous maps $\alpha_1:B\to \mathrm{GL}(r,\CC)$, $\alpha_2:\Omega\to \mathrm{GL}(r,\CC)$ such that, when restricted to $\Sigma$, $\alpha_1\alpha=\alpha_2$.
\end{lemma}
\begin{proof}
Let $E$ be the vector bundle on the unit ball in $\RR^n$
obtained by clutching  trivial  bundles on $B$ and $\Omega$ via $\alpha$,
\[ E = \Omega\times \CC^r\sqcup B\times \CC^r/\sim\qquad (p,v)\sim (p,\alpha(p)v)\quad p\in \Sigma\]
Because the unit ball is contractible, $E$ can be trivialized.
A trivialization of $E$ amounts to a choice of continuous maps $\alpha_1:B\to \mathrm{GL}(r,\CC)$, $\alpha_2:\Omega\to \mathrm{GL}(r,\CC)$
such that $\alpha_1(p)\alpha(p)v=\alpha_2(p)v$ for all $p\in \Sigma$, $v\in \CC^r$.

\end{proof}



\section{Proof of the index theorem}

Bott periodicity is the following statement about the homotopy groups of $\mathrm{GL}(n,\CC)$ \cite{Bo59},
\[ \pi_j\,\mathrm{GL}(n,\CC)\cong\left\{
\begin{array}{ll}  \ZZ & j\;\mbox{odd}\\ 0& j\;\mbox{even} \end{array}
\right.\]
\[j=0,1,2,\dots, 2n-1\]

\begin{proposition}\label{sphere2}
Let $S^m$ be an odd dimensional  sphere.
If $\alpha:S^m\to \mathrm{GL}(r,\CC)$ is a continuous map
then
\[  \ind T_\alpha = \ch(\alpha)[S^m]\]
\end{proposition}
\begin{proof}
Assume $r\ge (m+1)/2$, so that $\pi_m\mathrm{GL}(r,\CC)\cong \ZZ$. 
If not, replace $r$ with $r'\ge (m+1)/2$, and $\alpha$ with $\alpha\oplus I_{r'-r}$.

Due to homotopy invariance of the index, the analytic and topological index determine maps
\[ \phi:\pi_m\,\mathrm{GL}(r,\CC)\to \ZZ\qquad \phi([\alpha]) = \mathrm{Index}\,T_\alpha\]
\[ \psi:\pi_m\,\mathrm{GL}(r,\CC)\to \QQ\qquad \psi([\alpha]) = \ch(\alpha)[S^m]\]
Note that if $\alpha_1,\alpha_2:S^m\to \mathrm{GL}(r,\CC)$ represent two elements in $\pi_m \mathrm{GL}(r,\CC)$,
 their product in $\pi_m \mathrm{GL}(r,\CC)$ can be represented by  the map $p\mapsto \alpha_1(p)\alpha_2(p)$.
Since
\[ \mathrm{Index}\,T_{\alpha_1\alpha_2}=\mathrm{Index}\,T_{\alpha_1}+\mathrm{Index}\,T_{\alpha_2}
\qquad \ch(\alpha_1\alpha_2)=\ch(\alpha_1)+\ch(\alpha_2)\]
it follows that $\phi:\ZZ\to \ZZ$ and $\psi:\ZZ\to \QQ$ are  homomorphisms of abelian groups.

By Theorem \ref{thm:Noether}, for $f:S^1\to \CC\setminus \{0\}, f(z)=z^{-1}$,
\[ \mathrm{Index}\,T_f=\ch(f)[S^1]=1\]
Then by Lemma \ref{sphere}
there exists $\alpha:S^m\to \mathrm{GL}(r,\CC)$ with 
\[ \mathrm{Index}\,T_\alpha=\ch(\alpha)[S^m]=1\]
This implies that  $\phi([\alpha]) = \psi([\alpha])$
for all $[\alpha]\in  \pi_m \mathrm{GL}(r,\CC)\cong \ZZ$.

\end{proof}

\begin{remark}
For an alternate approach to this proof, using Bott periodicity combined with a direct calculation, see Venugopalkrishna \cite{V72}.
\end{remark}

\begin{proof}[Proof of Theorem \ref{Thm}]
By Lemma \ref{sphere2}, Theorem \ref{Thm} holds for odd dimensional spheres $M=S^m$.
Then by Lemma \ref{sphere}, Theorem \ref{Thm} holds for all odd dimensional spin$^c$ manifolds $M$.

\end{proof}

\section{Boutet de Monvel's theorem}\label{sec:BdM}

Let $\tilde{\Omega}$ be a complex analytic manifold, and $\Omega^0\subset \tilde{\Omega}$ a relatively compact open submanifold with  smooth boundary $M=\partial \Omega^0$.
Choose a  defining function of the boundary $\rho:\tilde{\Omega}\to \RR$ with  $\Omega^0=\rho^{-1}((-\infty,0))$, $\Omega=\rho^{-1}((-\infty,0])$, $M=\rho^{-1}(0)$, and $d\rho(p)\ne 0$ for all $p\in M$.
The boundary of $\Omega^0$  is called  strictly pseudoconvex if for every $p\in M$ and every holomorphic vector $w$ that is tangent to $M$,
\[w\in T^{1,0}_pM = T^{1,0}_p\tilde{\Omega}\cap (T_pM\otimes \CC)\] 
we have 
\[\partial\bar{\partial}\rho(w,\bar{w})> 0\qquad \text{if}\;w\ne 0\]
Strict pseudoconvexity is biholomorphically invariant.
A domain $\Omega^0$ in $\CC^n$ that is strictly convex in the Euclidean sense is strictly pseudoconvex.

The Hardy space  $H^2(M)$ is the space of $L^2$-functions on $M$ that extend to a holomorphic function on $\Omega$.
 The Szeg\"o projection $S$ is the orthogonal projection 
 \[ S:L^2(M)\to H^2(M)\]
For a continuous map $\alpha:M\to \mathrm{GL}(r,\CC)$, let $\mathcal{M}_\alpha$ be the corresponding multiplication operator 
on  $L^2(M)\otimes \CC^r$.
The Toeplitz operator $\mathfrak{T}_\alpha$ is the composition of $\mathcal{M}_\alpha$
with $S\otimes I_r$,
\[ \mathfrak{T}_\alpha =(S\otimes I_r)\mathcal{M}_\alpha : H^2(M)\otimes \CC^r\to H^2(M)\otimes \CC^r\]
$\mathfrak{T}_\alpha$ is a bounded Fredholm operator.
Boutet de Monvel's theorem is (Theorem 1 in \cite{Bo79}):
\begin{theorem}\label{BdM}
\[  \ind \mathfrak{T}_\alpha = (\ch(\alpha)\cup \Td(M))[M]\]
\end{theorem}
\begin{proof}
The complex manifold $\Omega$ is a spin$^c$ manifold,
and so is its boundary $M$.
\footnote{A strictly pseudoconvex boundary $M$ is a contact manifold, with contact $1$-form $(\partial\rho-\bar{\partial}\rho)|M$.
The spin$^c$ structure of $M$ as the boundary of $\Omega$ agrees with its spin$^c$ structure as a contact manifold.}
Using the projection $P_+$ onto the positive space of the Dirac operator of $M$, as  in the introduction of this paper,
we form a Toeplitz operator
\[ T_\alpha = (P_+\otimes I_r)\mathcal{M}_\alpha:L^2_+(M,S)\otimes \CC^r\to L^2_+(M,S)\otimes \CC^r\]
Theorem \ref{BdM} follows from  Theorem \ref{Thm}
if we can show that
\[ \ind \mathfrak{T}_\alpha = \ind T_\alpha\]
This is Proposition 4.6 of \cite{BDT89}.

An outline of the proof given in \cite{BDT89} is as follows. According to Proposition \ref{propd} above,
\[ \ind T_\alpha = \ind T_{\tilde{\alpha}}\]
where $\tilde{\alpha}:\Omega\to M_r(\CC)$ is any continuous function whose restriction to 
the boundary $M=\partial\Omega$ is $\alpha$, and $T_{\tilde{\alpha}}$ is the Toeplitz operator,
\[ T_{\tilde{\alpha}} : \Hardy\otimes \CC^r\to \Hardy\otimes \CC^r\]
Here $\Hardy$ is the null-space of the maximal closed extension of $D_\Omega$, as in section \ref{sec3}.

Following section 3 of \cite{BDT89}, we  now choose a different domain for the Dirac operator of $\Omega$,
using  $\bar{\partial}$-Neumann boundary conditions.
The Dirac operator of $\Omega$ is the assembled Dolbeault complex,
\[\bar{\partial}+\bar{\partial}^*:C^\infty(\Omega,S^+)\to C^\infty(\Omega,S^-)\]
with positive and negative  spinor bundles
\[ S^+= \bigoplus_{k\;\text{even}}\Lambda^kT^{0,1}\Omega\qquad S^-= \bigoplus_{k\;\text{odd}}\Lambda^kT^{0,1}\Omega\] 
Here $C^\infty(\Omega, S^{+/-})$ denotes the space of spinors that are smooth up to and including the boundary. Let 
\[ \mathcal{A}^k := \{u\in C^\infty(\Omega,\Lambda^kT^{0,1}\Omega)\;\mid\; \iota_{ \bar{\partial} \rho} u=0\;\text{on}\;M=\partial\Omega\}\]
where, as usual, $\iota$ denotes contraction.
Let $D_N$ denote the closure of $\bar{\partial}+\bar{\partial}^*$ restricted to $\bigoplus_{k\;\text{even}} \mathcal{A}^k$. Then
\[ D_\Omega\subseteq D_N\subseteq \bar{D}_\Omega\]
By the change of domain principle of section 2 of \cite{BDT89},
the index of any Toeplitz operator obtained by using the projection onto the null space $\Hardy\otimes \CC^r$,
is equal to the index of the Toeplitz operator obtained by using projection onto the null space $\mathcal{N}(D_N)\otimes \CC^r$ (Proposition 3.3 in \cite{BDT89}).

Let $\Box=D_N^*D_N$ be the complex Laplacian acting on $(0,k)$-forms with $k$ even,
likewise  $\Box=D_ND_N^*$ acting on $(0,k)$-forms with $k$ odd. 
By a result of J.\ Kohn \citelist{\cite{Ko63} \cite{Ko64}}, if the boundary of $\Omega$ is strictly pseudoconvex, the operator $\Box$ has compact resolvant for $k\ne 0$. 
This then implies that $\mathcal{N}(D_N^*)$ is finite dimensional,
and $\mathcal{N}(D_N)$ is at most a finite dimensional perturbation of 
\[ H^\omega(\Omega) = \{u\in L^2(\Omega)\;\mid\; \bar{\partial}u=0\}\]
The projection $L^2(\Omega)\to H^\omega(\Omega)$ is the Bergmann projection.
The index of any Toeplitz operator obtained by using projection onto the null space $\mathcal{N}(D_N)\otimes \CC^r$ 
is equal to the index of the Toeplitz operator obtained using the Bergmann projection.

Finally, the transition from the Bergmann projection (on $\Omega$) to the Szeg\"o projection $S$ (on $M=\partial\Omega$) is done by the same argument as in the proof of Proposition \ref{propd}.

In summary, the proof is
\[ P_+
\rightsquigarrow \text{Calderon}
\rightsquigarrow \mathcal{N}(\bar{D}_\Omega)  
\rightsquigarrow \mathcal{N}(D_N)
\rightsquigarrow \text{Bergmann}
\rightsquigarrow \text{Szeg\"o}
\]

\end{proof}

\begin{remark}
A crucial step in the proof of Theorem \ref{BdM} is the passage from spinors  to functions (i.e. the step $\mathcal{N}(D_N)
\rightsquigarrow \text{Bergmann}$),
which is done by applying a result of J.\ Kohn.
Compare this to the sheaf theoretic result that, if $\Omega$ is strictly pseudoconvex, then the sheaf cohomology $H^k(\Omega^0,\mathcal{O})$ is a finite dimensional complex vector space for $k>0$ (Proposition 4 in \cite{Gr58}).
Here $\mathcal{O}$ denotes the structure sheaf (germs of holomorphic functions) of $\Omega^0$. 
$H^0(\Omega^0,\mathcal{O})$ is the space of holomorphic functions on $\Omega^0$.
$H^k(\Omega^0,\mathcal{O})$
identifies with the $k$-th homology of the Dolbeault complex.
\end{remark}

\bibliographystyle{abbrv}
\bibliography{MyBibfile}

\end{document}